\documentclass[12pt]{article}

\newtheorem{definition}{Definition}

\newtheorem{proposition}{Proposition}

\usepackage{epsfig}
\usepackage{url}
\usepackage{amssymb}
\usepackage{amsmath}
\usepackage{amsfonts}

\def\Dbar{\leavevmode\lower.6ex\hbox to 0pt{\hskip-.23ex \accent"16\hss}D}

\def\bZ{{\mbox{\bf Z}}}

\newcommand{\nc}{\newcommand}
\nc{\cP}{{\cal P}}

\begin{document}

{\bf\LARGE
\begin{center}
Goethals--Seidel difference families with symmetric
or skew base blocks
\end{center}
}

{\Large
\begin{center}
Dragomir {\v{Z}}. {\Dbar}okovi{\'c}\footnote{University of Waterloo,
Department of Pure Mathematics and Institute for
Quantum Computing, Waterloo, Ontario, N2L 3G1, Canada
e-mail: \url{djokovic@uwaterloo.ca}}, Ilias S.
Kotsireas\footnote{Wilfrid Laurier University, Department of Physics
\& Computer Science, Waterloo, Ontario, N2L 3C5, Canada, e-mail:
\url{ikotsire@wlu.ca}}
\end{center}
}

\begin{abstract}
We single out a class of difference families which is widely
used in some constructions of Hadamard matrices and
which we call Goethals--Seidel (GS) difference families.
They consist of four subsets (base blocks) of a finite
abelian group of order $v$, which can be used to construct
Hadamard matrices via the well-known Goethals--Seidel array.
We consider the special class of these families in
cyclic groups, where each base block is either symmetric or skew.

We omit the well-known case where all four blocks are symmetric.
By extending previous computations by several authors,
we complete the classification of GS-difference families
of this type for odd $v<50$.
In particular, we have constructed the first examples of
so called good matrices, G-matrices and best matrices of
order 43, and good matrices and G-matrices of order 45.

We also point out some errors in one of the cited references.
\end{abstract}



\section{Introduction}

The well-known Goethals--Seidel array (GS-array)
\begin{equation} \label{GS-matrix}
\left[ \begin{array}{cccc}
Z_0 & Z_1R & Z_2R & Z_3R \\
-Z_1R & Z_0 & -Z_3^T R & Z_2^T R \\
-Z_2R & Z_3^T R & Z_0 & -Z_1^T R \\
-Z_3R & -Z_2^T R & Z_1^T R & Z_0
\end{array} \right],
\end{equation}
is a powerful tool for the construction of Hadamard matrices
and in particular those of skew type. The four blocks $Z_i$
of size $v$, needed to plug into the GS-array to obtain a
Hadamard matrix of order $4v$, are usually constructed by
using a suitable difference family in a finite abelian group
$G$ of order $v$. (The matrix $R$ will be defined later in the case when this abelian group is cyclic.)

We recall that the subsets $X_1,X_2,\ldots,X_t\subseteq G$ form a {\em difference family} if for every $c\in G\setminus\{0\}$ there are exactly $\lambda$ ordered triples
$(a,b,i)\in G\times G\times\{1,2,\ldots,t\}$
such that $\{a,b\}\subseteq X_i$ and $a-b=c$.
We say that $(v;k_1,k_2,\ldots,k_t;\lambda)$ is the {\em set of
parameters} of that family, where $k_i=|X_i|$ is the cardinality
of $X_i$. In the Hadamard matrix community these difference families are better known as {\em supplementary difference sets}, see e.g. \cite[Definition 1.5, p. 281]{JS:1972}. Note that we must have

\begin{equation} \label{lambda}
\sum_{i=1}^t k_i(k_i-1)=\lambda(v-1).
\end{equation}

We remark that if in a difference family $(X_i)$ we replace one of the blocks $X_i$ by its complement $G\setminus X_i$, then we
still have a difference family; its parameter set is obtained
from the original one by replacing $k_i$ with $v-k_i$ and
$\lambda$ by $\lambda+(v-2k_i)$.

The difference families that we need have $t=4$,
i.e., they are ordered quadruples of four base blocks
$X_1,X_2,X_3,X_4$, such that their parameter sets
satisfy the condition

\begin{equation} \label{sum-k}
\sum_{i=1}^4 k_i=\lambda+v.
\end{equation}

We say that the parameter sets, with $t=4$, satisfying the
conditions (\ref{lambda}) and (\ref{sum-k}) are
{\em GS-parameter sets} and that the
corresponding difference families are
{\em GS-difference families}.
While the condition (\ref{lambda}) is necessary for the
existence of the difference family with the above mentioned
parameter set, the two additional conditions that $t=4$ and
(\ref{sum-k}) are required for the construction of Hadamard
matrices by using the GS-array. Indeed, assuming that $t=4$,
the condition (\ref{sum-k}) is equivalent
to the condition ({sum-A}) below, see
\cite[Lemma 1.20]{JS:1972}.

In this paper the group $G$ will be cyclic,
$G=\bZ_v=\{0,1,\ldots,v-1\}$, $v>1$,
and the difference families will be of GS-type.
Since we can permute the $X_i$ and replace any $X_i$ with
its complement, we shall assume that

\begin{equation} \label{nejednakosti}
v/2\ge k_1\ge k_2\ge k_3\ge k_4\ge0.
\end{equation}

The blocks $Z_i$ will be circulant matrices while the matrix
$R$ in (\ref{GS-matrix}) will be the back-circulant
identity matrix of order $v$,
$$
\label{Matrix-R}
R=\left[ \begin{array}{ccccc}
0 & 0 & \cdots & 0 & 1 \\
0 & 0 &        & 1 & 0 \\
\vdots &       &   &   \\
0 & 1 &        & 0 & 0 \\
1 & 0 &        & 0 & 0
\end{array} \right].
$$

We say that a subset $X\subseteq\bZ_v$ is {\em symmetric}
if $-X=X$, and that it is {\em skew} if $v$ is odd,
$|X|=(v-1)/2$ and $(-X)\cap X=\emptyset$.

We consider only the GS-difference families $(X_i)$,
$i=1,2,3,4$, in $\bZ_v$ such that each $X_i$ is either
symmetric or skew. We have four {\em symmetry types} for such
families: (ssss), (ksss), (kkss), (kkks). (The type (kkkk) will be dismissed below.) The letter ``s'' stands for ``symmetric'' and ``k'' for ``skew''. For instance, a difference family is of (kkss) type if two of the base blocks are symmetric and the other two are skew.

To any subset $X\subseteq\bZ_v$ we associate the binary sequence
(i.e., a sequence with entries $+1$ and $-1$) of length $v$, say
$(x_0,x_1,\ldots,x_{v-1})$, where $x_i=-1$ if and only
if $i\in X$. By abuse of language, we shall use the
symbol $X$ to denote also the binary sequence
associated to the subset $X$. Further, let $A_i$ be the circulant matrix having the sequence $X_i$ as its first row.

For any GS-difference family $(X_i)$ the four circulants
$A_i$ satisfy the equation
\begin{equation} \label{sum-A}
\sum_{i=1}^4 A_i^T A_i=4vI_v,
\end{equation}
where $I_v$ is the identity matrix of order $v$.
This equation is used to verify that, after setting
$Z_i=A_{i+1}$, $i=0,\ldots,3$, into the GS-array,
we obtain a Hadamard matrix. Thus the equality
(\ref{sum-A}) is an essential feature of the GS-families.
For the sake of convenience, let us say that four circulant
$\{\pm1\}$-matrices $(A_i)$ of order $v$ are {\em GS-matrices}
if they satisfy the equation (\ref{sum-A}).

Let $(A_i)$ be GS-matrices. If they have the symmetry type
(ssss) then they are the well-known Williamson matrices (with four symmetric circulants).
If they have symmetry type (ksss), with $A_1$ of skew type,
i.e. $A_1+A_1^T=2I_v$, then the quadruple consisting of the circulant $A_1$ and the back-circulants $B_i=A_iR$, $i=2,3,4$, are good matrices. Alternatively, for $i=2,3,4$ one can
take $B_i$ to be the back-circulant matrix having the same first row as $A_i$.
(Note that each back-circulant is a symmetric matrix.)

As the reader may not be familiar with the good matrices, let
us define them precisely. We say that two matrices $M,N$ of order $v$ are {\em amicable} if $MN^T=NM^T$. Let $(M_i)$ be a quadruple
of $\{\pm1\}$-matrices of order $v$. We say that the $M_i$ are
{\em Willamson type matrices} if they are pairwise amicable and satisfy the equation (\ref{sum-A}). Finally, we say that the matrices $M_i$ are {\em good matrices} if they are Williamson type matrices, one of which is of skew type and the other three symmetric.

The G-matrices over a finite abelian group have been introduced
long ago by J. Seberry Walis \cite[p. 154]{JS75}. (The condition
(iii) of that definition is redundant as any two matrices
of type 1 commute.) In the case that we consider, the group $G$ is cyclic and G-matrices are simply the GS-matrices having the  symmetry type (kkss).

Similarly, the GS-matrices of symmetry type (kkks) have been
given the name of {\em best matrices} (see \cite{GKS:LAMA:2001}).

By eliminating the parameter $\lambda$ from the
equations (\ref{lambda}) and (\ref{sum-k}),
we obtain that
\begin{equation} \label{sum-kv}
\sum_{i=1}^4 (v-2k_i)^2=4v.
\end{equation}
It is also true that equations (\ref{lambda}) and (\ref{sum-kv})
imply the equation (\ref{sum-k}).

Hence, if all $X_i$ are skew then necessarily $v=1$.
For that reason we have omitted the (kkkk) type.
Note also that if $k_4=0$ then (\ref{sum-kv}) implies that
$v\le4$. Since the Williamson matrices are treated extensively
in \cite{HKT}, we omit the (ssss) type as well.

In order to find GS-matrices of a given order $v$, one has to
compute first the possible GS-parameter sets. This is easily
accomplished by solving the Diophantine equation (\ref{sum-kv}) subject to the conditions (\ref{nejednakosti}).

Our main objective is to extend and complete the classification
(up to equivalence) of the GS-difference families of types
(ksss), (kkss) and (kkks) for odd $v<50$. See Table 1 in
section \ref{summary} for the summary of old and new results
in this range.

\section{GS-parameter sets}

The following proposition follows immediately from the
equation (\ref{sum-kv}).

\begin{proposition} \label{param-1}
The GS-parameter sets $(v;k_1,k_2,k_3,k_4;\lambda)$ with
$v/2\ge k_1\ge k_2\ge k_3\ge k_4\ge 0$ are parametrized
by four-odd-square decompositions
$4v=\sum_{i=1}^4 s_i^2$ when $v$ is odd, and
by four-square decompositions
$v=\sum_{i=1}^4 s_i^2$ when $v$ is even.
In both cases we require that $0\le s_1\le s_2\le s_3\le s_4$.
\end{proposition}

The parameter sets needed for the construction of
GS-difference families of symmetry types (ksss), (kkss) and
(kkks) always exist. We treat the three cases separately.

\begin{proposition} \label{stav-0}
For a given odd $v\ge1$, there is at least one GS-parameter set
with $(v-1)/2=k_1\ge k_2\ge k_3\ge k_4$.
\end{proposition}
\begin{proof}
By setting $k_1=(v-1)/2$ the equation (\ref{sum-kv}) reduces to
$$
\sum_{i=2}^4 (v-2k_i)^2 = 4v-1.
$$
Since $v$ is odd, $4v-1\equiv 3 \pmod{8}$. It is known that
every positive integer congruent to 3 modulo 8 is a sum
of three odd squares. Indeed, this claim is equivalent to the
well-known theorem of Gauss that every positive integer is a sum of at most three triangular numbers \cite{Duke:1997}.
Thus there exist integers $k_2,k_3,k_4$ satisfying the above
equation and such that $(v-1)/2\ge k_2\ge k_3\ge k_4\ge0$.
Consequently, the $k_i$, $i=1,\ldots,4$, satisfy the equation
(\ref{sum-kv}) which implies that $\lambda\ge0$.
\end{proof}

In the case (ksss) there are only finitely many known
GS-difference families, namely for odd $v\le39$
and $v=127$. The cases $v=43,45$ are new and appear here
for the first time. Our exhaustive search did not find
any GS-difference families of type (ksss) for $v=41,47,49$.

In the case (kkss) there is an infinite series of
GS-difference families constructed by E. Spence
\cite{Spence:1977}.

For type (kkss), the GS-parameter sets are
given by the following proposition.

\begin{proposition} \label{stav-1}
The GS-parameter sets with $(v-1)/2=k_1=k_2\ge k_3\ge k_4$ exist if and only if $2v-1$ is a sum of two squares. They have
the form
$$
(v; k_1=k_2=(v-1)/2,(v-r+s)/2,(v-r-s)/2;v-r-1),
$$
where $r,s$ are integers, $r>s\ge 0$ and $r^2+s^2=2v-1$.
\end{proposition}
\begin{proof}
Assume first that $k_1=k_2=(v-1)/2$. The equation (\ref{sum-k})
gives $\lambda=k_3+k_4-1$ and from the equation (\ref{lambda})
we obtain that
$$
k_3(k_3-1)+k_4(k_4-1)=(v-1) \left( k_3+k_4-\frac{v-1}{2} \right).
$$
This equation can be written as $(v-2k_3)^2+(v-2k_4)^2=2(2v-1)$.
This implies that $2v-1$ is a sum of two squares.

Conversely, assume that $2v-1=r^2+s^2$ where $r\ge s\ge0$ are
integers. In fact we must have $r>s$. If $v=3$ we have $r=1$ and $s=0$. If $v=5$ we have $r=3$ and $s=0$. So, the assertion holds
in these two cases. We may assume that $v>5$. We claim that
$r\le(v-1)/2$. Otherwise we would have $2r\ge v+1$ and $4(2v-1)=4r^2+4s^2\ge(v+1)^2$ which gives the contradiction
$(v-1)(v-5)\le0$. As $r$ and $s$ have different parity, we can set $k_3=(v-r+s)/2$, $k_4=(v-r-s)/2$ and $\lambda=v-r-1\ge0$.
It is easy to verify that we obtain a valid GS-parameter set.
\end{proof}

For type (kkks), the GS-parameter sets are
given by the following proposition.

\begin{proposition} \label{stav-2}
The GS-parameter sets with $(v-1)/2=k_1=k_2=k_3\ge k_4$ have
the form
$$
(v=r^2+r+1; k_1=k_2=k_3=r(r+1)/2,~ k_4=r(r-1)/2;
 \lambda=r^2-1),
$$
where $r$ is a positive integer.
\end{proposition}
\begin{proof}
Set $v=2s+1$, and so $k_1=k_2=k_3=s$.
From (\ref{sum-k}) we obtain that $k_4=\lambda-s+1$.
Since $\sum k_i(k_i-1)=\lambda(v-1)=2s\lambda$, we deduce
that $\lambda$ satisfies the quadratic equation
$\lambda^2-(4s-1)\lambda+4s(s-1)=0$. The discriminant
$8s+1$ must be an odd square, say $(2r+1)^2$. Thus,
$s=r(r+1)/2$. The two roots of the quadratic equation
are $r^2-1$ and $2r+r^2$. We omit the second root because
it gives that $k_4>s=(v-1)/2$. This completes the proof.
\end{proof}

We observe that the formula $v=r^2+r+1$ implies that
$4v-3$ is an odd square (see also \cite{GKS:LAMA:2001}).

It is not known whether for each of these parameter sets
there exists a GS-difference family of type (kkks).
They exist for the first six values of $r$,
$r=1,2,\ldots,6$.

\section{Elementary transformations and equivalence relation}

Let $(X_i)$ be a difference family over a finite abelian
group $G$ of odd order $v$ (written additively). We fix a
GS-parameter set such that (\ref{nejednakosti}) holds.
The set of GS-families with this parameter set is invariant under the following elementary transformations:

(a) For some $i$ replace $X_i$ by a translate $g+X_i$,
$g\in G$.

(b) For some $i$ replace $X_i$ by $-X_i$.

(c) For all $i$ replace $X_i$ by its image $\alpha(X_i)$
under an automorphism $\alpha$ of $G$.

(d) Exchange $X_i$ and $X_j$ provided that $|X_i|=|X_j|$.

In the case $G=\bZ_v$, $\alpha$ is just the multiplication
modulo $v$ by some integer $k$ relatively prime to $v$.

\begin{definition} \label{ekv}
We say that two difference families with the same parameter set satisfying (\ref{nejednakosti}) are {\em equivalent} if one can be transformed to the other by a finite sequence of elementary
transformations.
\end{definition}

Note that the symmetry properties of the base blocks
$X_i$ may be destroyed by operations (a), and changed
by operations (d). Nevertheless we shall use the above
definition of equivalence for the classification of
GS-difference families with symmetry.

The definition of equivalence for the special case of
cyclic GS-difference families of type (ksss), (kkss), and (kkks)
given in \cite{GK:JCMCC:2002,GKS:LAMA:2001,GKS:CompStat:2002}
uses only elementary transformations of type (c).
Apparently their GS-difference families (also known as
supplementary difference sets) are non-ordered quadruples,
in which case they do not need the elementary operations
of type (d). For convenience, we shall refer to their
equivalence classes as {\em small classes}.
If $X_i$ is symmetric then $-X_i=X_i$ but this fails in
the case when $X_i$ is skew. Hence, for GS-difference families
of type (ksss) the elementary operation (b) can be viewed as a  special case of (c). However, this is not the case for
GS-difference families of type (kkss) or (kkks).
For that reason, our equivalence class may contain more than
one small class.
In the (kkss) case an equivalence class may contain
two small classes, and in the (kkks) case it may contain
up to four small classes.

Table 1 of \cite{GK:JCMCC:2002} contains the list of
representatives of the small equivalence classes of
supplementary difference sets (i.e., GS-difference families)
of type (kkss) for all odd $v\le 33$. It should be used
with care as it contains several errors.
For example, let us consider the very first
item in that table, namely the parameter set (3;1,1,2,0;1).
It is claimed there that there is only one small equivalence
class, namely the one with the base blocks
$S_0=\{1\}$, $S_1=\{1\}$, $S_2=\{1,2\}$, $S_3=\emptyset$.
However, if we replace $S_0$ with $-S_0=\{2\}$ then,
according to the definition given in
\cite[p. 207]{GK:JCMCC:2002},
the families $-S_0,S_1,S_2,S_3$ and $S_0,S_1,S_2,S_3$ are
not equivalent. Note that they are equivalent according to
our Definition \ref{ekv}. For more examples see
our comment at the end of subsection \ref{33} of the appendix.

\section{Algorithmic issues}

All computations for GS-difference families $(X_1,X_2,X_3,X_4)$
with parameters $(v;k_1,k_2,k_3,k_4;\lambda)$, $v$ odd, in this paper have been performed via a two-phase algorithmic scheme.
The first phase is the data collection phase.
The second phase is the matching phase. We provide a more detailed
description of these two phases below.

\subsection{Data collection}

During the data collection phase, we collect four data files
$F1,F2,F3,F4$, corresponding to the four base blocks
$X_1,X_2,X_3,X_4$ of the GS-difference family.
Each data file contains positive integers arranged in $k$ columns
where $k=(v-1)/2$. Each data file is typically made up of several millions of rows, and corresponds to an exhaustive search for sequences arising from the corresponding base block.
The usual power spectral density filtering (see \cite{FGS:2001}) is applied, during the data collection phase.
The positive integers in the four data files represent the multiplicities of the various differences arising from the corresponding base blocks.
More specifically, a combinatorial generation algorithm is used to generate exhaustively all possible base blocks that are not eliminated from the power spectral density filtering.
If we denote by $X_{1,i}$ the $i$th candidate for the first base block $X_1$, the $i$th row of file $F1$ contains a list of how often each group
element is represented as a difference $g_1 - g_2$ where
$g_1, g_2 \in X_{1,i}$. The same principle applies to files $F2,F3,F4$ as well.
Here is a schematic representation of the four data files:
$$
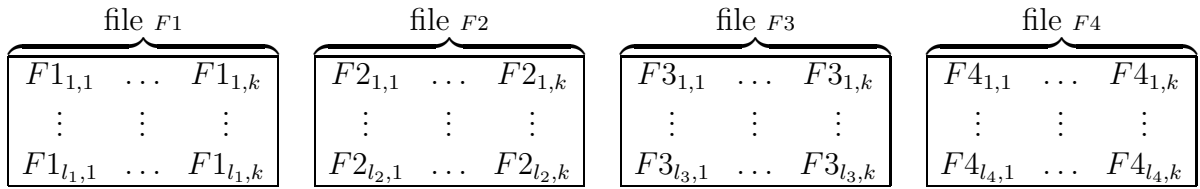
\overbrace{
    \begin{array}{|ccc|}
    \hline
    F1_{1,1}   & \ldots  & F1_{1,k} \\
    \vdots     & \vdots  & \vdots \\
    F1_{l_1,1} & \ldots  & F1_{l_1,k} \\
    \hline
    \end{array}
    }^{\mbox{file } F1}
        \quad
\overbrace{
    \begin{array}{|ccc|}
    \hline
    F2_{1,1}   & \ldots  & F2_{1,k} \\
    \vdots     & \vdots  & \vdots \\
    F2_{l_2,1} & \ldots  & F2_{l_2,k} \\
    \hline
    \end{array}
    }^{\mbox{file } F2}
        \quad
\overbrace{
    \begin{array}{|ccc|}
    \hline
    F3_{1,1}   & \ldots & F3_{1,k} \\
    \vdots     & \vdots & \vdots \\
    F3_{l_3,1} & \ldots & F3_{l_3,k} \\
    \hline
    \end{array}
    }^{\mbox{file } F3}
        \quad
\overbrace{
    \begin{array}{|ccc|}
    \hline
    F4_{1,1}   & \ldots  & F4_{1,k} \\
    \vdots     & \vdots  & \vdots \\
    F4_{l_4,1} & \ldots  & F4_{l_4,k} \\
    \hline
    \end{array}
    }^{\mbox{file } F4}
$$
where $l_1, l_2, l_3, l_4$ are the total numbers of lines of the files $F1, F2, F3, F4$.
Due to the fact that the differences $d$ and $v-d$ have
the same multiplicity, we record these multiplicities only
for $d=1,2,\ldots,k$ where $k=(v-1)/2$.

\subsection{Matching}

During the matching phase, we process the four data files $F1, F2, F3, F4$,
with the aim to find quadruples of lines (one in each file), so that
the $k$ element-wise sums are all equal to the constant $\lambda$, defined in the beginning of the Introduction.
More specifically, we are looking for a line $\ell_{1}$ in file $F1$, a line $\ell_{2}$ in file $F2$, a line $\ell_{3}$ in file $F3$, a line $\ell_{4}$ in file $F4$,
such that the following $k$ conditions are satisfied:
$$
    \left\{
    \begin{array}{c}
        F1_{\ell_{1},1} + F2_{\ell_{2},1} + F3_{\ell_{3},1} + F4_{\ell_{4},1} = \lambda \\
        \vdots \\
        F1_{\ell_{1},k} + F2_{\ell_{2},k} + F3_{\ell_{3},k} + F4_{\ell_{4},k} = \lambda \\
    \end{array}
    \right.
$$
It may be the case that such quadruples do not exist,
which proves the non-existence of the corresponding GS-difference families.
If such quadruples do exist, then we compute all of them
by using a bins-based technique and a parallel backtracking algorithm \cite{P:1984}.
We use two different algorithms and two different implementations, to validate the results.

The bins-based technique works by performing a pre-processing step which consists in calculating
the different values that appear on the first column of each of the four data files, and subsequently splitting
the four data files in smaller parts, so that each part contains only those rows that start with
a specific number. Since we know that numbers in the first column must sum to $\lambda$, we are able to
compute all possible cases for which this can happen, given all the possible values appearing in the first columns of $F1,F2,F3,F4$.
This process typically establishes a few hundreds of possible cases $C_i$, each of which can then be processed in parallel.
A typical case $C_i$ is identified by four numbers $n_1,n_2,n_3,n_4$ such that $n_1 + n_2 + n_3 + n_4 = \lambda$ and four files
$F1_{C_i},F2_{C_i},F3_{C_i},F4_{C_i}$, such that $F1_{C_i}$ contains all lines in $F1$ that start with $n_1$ and similarly for the other three files.
It is also possible that the number of cases is zero, and in that case it is clear the problem does not have a solution.
If the number of cases is not zero, then for each case $C_i$ we examine the total numbers of lines in the corresponding four files,
$F1_{C_i},F2_{C_i},F3_{C_i},F4_{C_i}$ and if the following condition is satisfied
$$
    ( \# \mbox{lines in} \, F1_{C_i} ) \times ( \# \mbox{lines in} \, F4_{C_i} ) < 10^7
    \,\, \& \,\,
    ( \# \mbox{lines in} \, F2_{C_i} ) \times ( \# \mbox{lines in} \, F3_{C_i} ) < 10^7
$$
then case $C_i$ is dealt with by a serial C program.
Note that this condition is checked, once we have rearranged the files $F1_{C_i},F2_{C_i},F3_{C_i},F4_{C_i}$ in increasing order,
according to their total numbers of lines.
If the above condition is not satisfied for case $C_i$, then the four files $F1_{C_i},F2_{C_i},F3_{C_i},F4_{C_i}$
are split in smaller parts, by using the bins based on the second column. The threshold value $10^7$ has been determined experimentally.

\section{Summary of computational results and conclusions}
\label{summary}

There are several known infinite classes of difference
families with four base blocks which can be used
to construct Hadamard matrices via the Goethals--Seidel
array. For that reason we refer to these families as
GS-difference families. As an example let us mention one
such family constructed in \cite{XXSW:2005} which is especially
interesting as it gives skew Hadamard matrices. It was observed
recently (see \cite{Djokovic:SpecMatC:2015}) that by using
a different array (different from the GS-array) this infinite series gives also symmetric Hadamard matrices of the same orders.
We are not aware of any GS-parameter set for which there is
no cyclic difference family having the parameters specified by
this set.

The cyclic GS-difference families with parameter set
$(v;k_1,k_2,k_3,k_4;\lambda)$ of symmetry type
(ksss), (kkss) and (kkks) provide quadruples of matrices
of order $v$ known as good matrices, G-matrices and
best matrices, respectively. Some sporadic examples
of such matrices are given, and systematic exhaustive searches
for some small values of $v$ have been carried out in
\cite{Djokovic:JCMCC:1993,Djokovic:OperMatr:2009,
GK:JCMCC:2002,GKS:LAMA:2001,GKS:CompStat:2002,Spence:1977,
Szekeres:1988}.

There are only finitely many good matrices that have been
constructed so far. Their orders are odd integers up to 39
and the integers 43, 45 and 127. The ones of orders
43 and 45 appear first in this paper. The order 41 is
the smallest odd integer for which there are no GS-families
of (ksss) symmetry type.
Our exhaustive searches show that such difference families do not exist also in cyclic groups of orders 47 and 49.

There is an infinite series of circulant G-matrices due to
E. Spence \cite{Spence:1977}. This fact is not well-known.
For instance it is not mentioned in the paper
\cite{GK:JCMCC:2002} devoted to G-matrices. In that paper
the authors have classified the GS-difference families
of symmetry type (kkss) for odd $v\le 33$ and given a
few examples for $v=37,41$. In the appendix (see subsection
\ref{33}) we point out that this classification has some errors.
We have performed exhaustive searches for $v=37,41,43,45$ and 49.
We found that there are 7,3,4,1 equivalence classes
for the first four of these orders, and none for $v=49$.
There are no GS-parameter sets for G-matrices when
$v=35,39,47$.

The GS-parameter sets for best matrices (symmetry type
(kkks)) form an infinite series, see Proposition
\ref{stav-2}. Apparently there is no known infinite
series of best matrices so far. There is a simple necessary
arithmetic condition for the existence of best matrices
of order $v$, namely $4v-3$ must be an odd square.
For $v=3,7,13,21,31$ the small equivalence classes are
listed in \cite{GKS:LAMA:2001}. However, we warn the reader
that these lists may contain errors just like in the
case of G-matrices. For instance when $v=3$ there are
two small equivalence classes while only one is
recorded in \cite{GKS:LAMA:2001}. Our exhaustive
search for $v=43$ found that there are five equivalence
classes of best matrices.

A short summary of the known facts about the existence
of good matrices, G-matrices and best matrices for
odd orders $v<50$ is given in Table 1 below.
We write ``yes'' if the matrices exist and ``no''
if the exhaustive search did not find any such matrices.
The entry ``$\times$'' means that the parameter set is not
suitable for the symmetry type specified for the column.
The entries that we obtained in this work are shown in
boldface in Table 1.

$$ \begin{array}{ccccccccc c ccccccccc} \hline
v  &  k_1  &  k_2  &  k_3  & k_4 & \lambda  &
\mbox{ksss} & \mbox{kkss} & \mbox{kkks} & \quad &
v  &  k_1  &  k_2  &  k_3  & k_4 & \lambda  &
\mbox{ksss} & \mbox{kkss} & \mbox{kkks} \\ \hline
3 &  1 &  1 &  1  & 0 &  0 &
\mbox{yes} & \mbox{yes} & \mbox{yes} & &
31 &  15 &  15 &  15 &  10 & 24 &
\mbox{yes} & \mbox{yes} & \mbox{yes} \\
5 &  2 &  2 &  1  & 1 &  1 &
\mbox{yes} & \mbox{yes} & \times & &
  &  15 &  13 &  12 &  12 & 21 &
\mbox{yes} & \times  & \times \\
7 &  3 &  3 &  3  & 1 &  3 &
\mbox{yes} & \mbox{yes} & \mbox{yes} & &
33 &  16 &  16 &  15 &  11 & 25 &
\mbox{yes} & \mbox{yes} & \times \\
  &  3 &  2 & 2  & 2 & 2 &
\mbox{yes} & \times & \times & &
  &  16 &  16 &  13 &  12 & 24 &
\mbox{yes} & \mbox{yes} & \times \\
9 & 4 & 4 & 3 & 2 & 4 &
\mbox{yes} & \mbox{yes} & \times & &
  &  16 &  14 &  14 &  12 & 23 &
\mbox{yes} & \times & \times \\
11 & 5 & 4 & 4 & 3 & 5 &
\mbox{yes} & \times & \times & &
35  &  17 &  16 &  16 &  12 & 26 &
\mbox{yes} & \times & \times \\
13 & 6 & 6 & 6 & 3 & 8 &
\mbox{yes} & \mbox{no} & \mbox{yes} & &
   &  17 &  16 &  14 & 13 & 25 &
\mbox{yes} & \times & \times \\
   & 6 & 6 & 4 & 4 & 7 &
\mbox{yes} & \mbox{yes} & \times & &
37  &  18 &  18 &  16 &  13 & 28 &
\mbox{yes} & \mbox{yes} & \times \\
15 & 7 & 7 & 6 & 4 & 9 &
\mbox{yes} & \mbox{yes}  & \times & &
   &  18 & 15 &  15 & 15 & 26 &
\mbox{yes} & \times & \times \\
  & 7 & 6 & 5 & 5 & 8 &
\mbox{yes} & \times  & \times & &
39 & 19 & 18 &  17 & 14 & 29 &
\mbox{yes} & \times & \times \\
17 & 8 & 7 & 7 & 5 & 10 &
\mbox{yes} & \times & \times & &
   &  19 & 17 &  16 & 15 & 28 &
\mbox{yes} & \times & \times \\
19 & 9 & 9 & 7 & 6 & 12 &
\mbox{yes} & \mbox{yes}  & \times & &
41 & 20 & 20 & 16 & 16 & 31 &
\mbox{\bf no} & \mbox{yes} & \times \\
 & 9 & 7 & 7 & 7 & 11 &
\mbox{yes} & \times & \times & &
43 & 21 & 21 & 21 & 15 & 35 &
\mbox{\bf no} & \mbox{\bf yes} & \mbox{\bf yes} \\
21 & 10 & 10 & 10 & 6 & 15 &
\mbox{yes} & \mbox{yes} & \mbox{yes} & &
   & 21 & 21 & 18 & 16 & 33 &
\mbox{\bf no} & \mbox{\bf yes} & \times \\
   & 10 & 9 & 8 & 7 & 13 &
\mbox{yes} & \times & \times & &
   & 21 & 19 & 19 & 16 & 32 &
\mbox{\bf yes} & \times & \times \\
23 & 11 & 11 & 10 & 7 & 16 &
\mbox{yes} & \mbox{yes} & \times & &
   & 21 & 20 & 17 & 17 & 32 &
\mbox{\bf no} & \times & \times \\
25 & 12 & 12 & 9 & 9 & 17 &
\mbox{no} & \mbox{yes} & \times & &
45 & 22 & 22 & 21 & 16 & 36 &
\mbox{\bf yes} & \mbox{\bf yes} & \times \\
   & 12 & 11 & 11 & 8 & 17 &
\mbox{yes} & \times & \times & &
   & 22 & 21 & 19 & 17 & 34 &
\mbox{\bf yes} & \times & \times \\
   & 12 & 10 & 10 & 9 & 16 &
\mbox{yes} & \times & \times & &
   & 22 & 19 & 19 & 18 & 33 &
\mbox{\bf yes} & \times & \times \\
27 & 13 & 13 & 11 & 9 & 19 &
\mbox{yes} & \mbox{yes} & \times & &
47 & 23 & 22 & 22 & 17 & 37 &
\mbox{\bf no} & \times & \times \\
   & 13 & 12 & 10 & 10 & 18 &
\mbox{yes} & \times & \times & &
   & 23 & 21 & 19 & 19 & 35 &
\mbox{\bf no} & \times & \times \\
29 & 14 & 13 & 12 & 10 & 20 &
\mbox{yes} & \times & \times & &
49 & 24 & 24 & 22 & 18 & 39 &
\mbox{\bf no} & \mbox{\bf no} & \times \\
 & & & & & &
 & & & &
   & 24 & 22 & 21 & 19 & 37 &
\mbox{\bf no} & \times & \times \\
\end{array} $$
\begin{center} Table 1. Existence of GS-difference families with symmetry \end{center}

More details, including the listing of representatives
of equivalence classes of GS-difference families for $33\le v\le 49$, are provided in the appendix.

For applications of the good, best and G-matrices to
construction of the various combinatorial structures
see \cite{Djokovic:JCMCC:1993,GK:JCMCC:2002,
GKS:LAMA:2001,GKS:CompStat:2002}. Let us just mention two
of them. First, good matrices of order $v$ give the
complex Hadamard matrices of order $2v$.
Second, by Theorems 1 and 2 of \cite{GKS:LAMA:2001}, if
there exist best matrices of order $v$, then there exist
orthogonal designs
$OD(8v;1,1,2,2,4v-3,4v-3)$ and
$OD(4v;1,1,1,4v-3)$. In particular, this holds for $v=43$
as we have constructed best matrices of order 43. For the
definition of OD's see the same paper or \cite{SY:1992}.

\section{Acknowledgements}
The authors would like to thank the referees for their constructive comments, 
that contributed to improving the paper.
The authors wish to acknowledge generous support by NSERC, grant numbers 5285-2012 and 213992.
Computations were performed on the SOSCIP Consortium's Blue Gene/Q, computing platform.
SOSCIP is funded by the Federal Economic Development Agency of Southern Ontario, IBM Canada Ltd.,
Ontario Centres of Excellence, Mitacs and 14 academic member institutions.

\section{Appendix}

For the previous work on the classification of GS-families
of type (ksss) see \cite{GKS:CompStat:2002},
for the type (kkss) see \cite{GK:JCMCC:2002}, and
for (kkks) see \cite{GKS:LAMA:2001}. The paper
\cite{Djokovic:OperMatr:2009} gives additional GS-families
of all these types, as well as some multi-circulant GS-families
of type (ssss).

\noindent For each odd integer $v$ in the range $33\le v\le 49$ and
for each GS-parameter set $(v;k_1,k_2,k_3,k_4;\lambda)$
with $k_1=(v-1)/2 \ge k_2\ge k_3\ge k_4\ge0$ we list below
(or provide references for) the representatives of the
equivalence classes for GS-difference families
having symmetry type (ksss), (kkss) or (kkks).

\subsection{ $v=33$ \label{33} }

For the symmetry type (ksss), we have verified the claim made
in \cite{GKS:CompStat:2002} that there are 6, 4 and 5
equivalence classes of GS-difference families with
parameter sets $(33;16,16,15,11;25)$, $(33;16,16,13,12;24)$
and $(33;16,14,14,12;23)$, respectively.

On the other hand, our computations for type (kkss) do not
agree with those in \cite{GK:JCMCC:2002}.
We found that there are 20 equivalence classes. Their
representatives are listed below.

\begin{verbatim}
(kkss) type:

(a) (33;16,16,15,11;25)
[2,4,6,10,13,14,15,16,21,22,24,25,26,28,30,32],
[3,4,5,6,7,8,9,12,15,16,19,20,22,23,31,32],
[0,2,7,8,12,13,14,16,17,19,20,21,25,26,31],
[0,2,5,8,10,11,22,23,25,28,31].

(b) (33;16,16,15,11;25)
[1,3,5,6,11,12,14,16,18,20,23,24,25,26,29,31],
[2,4,5,9,11,12,13,15,16,19,23,25,26,27,30,32],
[0,1,2,3,5,11,12,16,17,21,22,28,30,31,32],
[0,8,9,12,15,16,17,18,21,24,25].

(c) (33;16,16,15,11;25)
[4,5,7,8,9,10,12,14,16,18,20,22,27,30,31,32],
[3,7,8,9,10,12,13,15,16,19,22,27,28,29,31,32],
[0,1,6,7,8,10,11,14,19,22,23,25,26,27,32],
[0,7,10,12,15,16,17,18,21,23,26].

(d) (33;16,16,15,11;25)
[3,4,5,6,7,9,10,12,14,16,18,20,22,25,31,32],
[1,2,4,5,6,7,8,13,14,16,18,21,22,23,24,30],
[0,1,3,4,9,10,11,14,19,22,23,24,29,30,32],
[0,3,4,7,9,12,21,24,26,29,30].

(e) (33;16,16,15,11;25)
[2,4,6,8,11,12,13,16,18,19,23,24,26,28,30,32],
[2,3,4,5,8,11,16,18,19,20,21,23,24,26,27,32],
[0,1,2,3,7,8,10,11,22,23,25,26,30,31,32],
[0,4,6,9,10,16,17,23,24,27,29].

(f) (33;16,16,15,11;25)
[2,4,5,10,12,16,18,19,20,22,24,25,26,27,30,32],
[2,5,9,11,12,13,14,16,18,23,25,26,27,29,30,32],
[0,1,3,6,7,10,11,12,21,22,23,26,27,30,32],
[0,7,8,12,15,16,17,18,21,25,26].

(g) (33;16,16,13,12;24)
[1,3,5,7,8,10,11,13,14,15,16,21,24,27,29,31],
[1,3,7,9,10,11,15,16,19,20,21,25,27,28,29,31],
[0,3,4,5,7,10,11,22,23,26,28,29,30],
[3,8,9,11,12,13,20,21,22,24,25,30].

(h) (33;16,16,13,12;24)
[2,3,7,9,10,12,14,15,16,20,22,25,27,28,29,32],
[1,7,8,10,12,16,18,19,20,22,24,27,28,29,30,31],
[0,2,3,4,5,11,14,19,22,28,29,30,31],
[5,7,8,11,12,16,17,21,22,25,26,28].

(i) (33;16,16,13,12;24)
[1,3,5,6,8,9,12,14,15,16,20,22,23,26,29,31],
[1,2,10,11,14,16,18,20,21,24,25,26,27,28,29,30],
[0,3,4,8,9,10,12,21,23,24,25,29,30],
[1,2,3,9,12,14,19,21,24,30,31,32].

(j) (33;16,16,13,12;24)
[1,3,5,7,8,11,14,15,16,20,21,23,24,27,29,31],
[1,2,5,8,10,12,13,14,15,16,22,24,26,27,29,30],
[0,4,5,6,12,14,15,18,19,21,27,28,29],
[1,2,3,4,7,14,19,26,29,30,31,32].

(k) (33;16,16,13,12;24)
[2,3,5,7,10,14,15,16,20,21,22,24,25,27,29,32],
[2,4,5,6,11,12,14,15,16,20,23,24,25,26,30,32],
[0,2,4,7,8,9,10,23,24,25,26,29,31],
[5,8,9,12,13,15,18,20,21,24,25,28].

(l) (33;16,16,13,12;24)
[3,4,5,6,8,10,12,13,15,16,19,22,24,26,31,32],
[2,4,5,8,9,10,12,13,14,15,16,22,26,27,30,32],
[0,1,2,3,9,11,15,18,22,24,30,31,32],
[3,6,7,11,12,14,19,21,22,26,27,30].

(m) (33;16,16,13,12;24)
[1,2,4,5,6,8,10,12,13,15,16,19,22,24,26,30],
[1,2,3,4,6,7,8,11,13,14,16,18,21,23,24,28],
[0,4,6,9,10,15,16,17,18,23,24,27,29],
[3,4,12,13,15,16,17,18,20,21,29,30].

(n) (33;16,16,13,12;24)
[1,5,7,8,11,13,15,16,19,21,23,24,27,29,30,31],
[2,3,4,5,7,8,9,11,12,14,16,18,20,23,27,32],
[0,7,8,9,10,13,14,19,20,23,24,25,26],
[3,4,6,8,9,16,17,24,25,27,29,30].

(o) (33;16,16,13,12;24)
[4,5,7,11,13,15,16,19,21,23,24,25,27,30,31,32],
[1,2,3,4,5,6,9,11,12,13,16,18,19,23,25,26],
[0,6,9,11,12,13,16,17,20,21,22,24,27],
[2,6,7,9,11,12,21,22,24,26,27,31].

(p) (33;16,16,13,12;24)
[2,3,6,11,13,15,16,19,21,23,24,25,26,28,29,32],
[1,3,4,5,9,16,18,19,20,21,22,23,25,26,27,31],
[0,3,5,9,12,13,14,19,20,21,24,28,30],
[4,6,9,10,15,16,17,18,23,24,27,29].

(q) (33;16,16,13,12;24)
[3,5,8,9,13,15,16,19,21,22,23,26,27,29,31,32],
[1,3,4,5,7,9,13,14,15,16,21,22,23,25,27,31],
[0,1,2,3,4,11,14,19,22,29,30,31,32],
[3,4,8,11,13,16,17,20,22,25,29,30].

(r) (33;16,16,13,12;24)
[3,4,6,7,9,12,13,14,16,18,22,23,25,28,31,32],
[2,3,4,5,7,9,10,11,12,15,16,19,20,25,27,32],
[0,1,2,4,6,13,16,17,20,27,29,31,32],
[3,4,7,9,15,16,17,18,24,26,29,30].

(s) (33;16,16,13,12;24)
[1,2,3,4,7,9,12,14,16,18,20,22,23,25,27,28],
[2,3,4,5,6,9,10,11,12,15,16,19,20,25,26,32],
[0,1,2,3,9,12,16,17,21,24,30,31,32],
[2,6,10,12,13,15,18,20,21,23,27,31].

(t) (33;16,16,13,12;24)
[2,6,8,10,11,12,13,16,18,19,24,26,28,29,30,32],
[1,3,4,5,7,8,9,10,12,15,16,19,20,22,27,31],
[0,7,8,10,12,13,16,17,20,21,23,25,26],
[4,5,6,8,13,14,19,20,25,27,28,29].
\end{verbatim}

The parameter set for the first 6 equivalence classes (a-f) is
(33;16,16,15,11;25) and for the remaining 14 classes (g-t) is  (33;16,16,13,12;24).
Each of the equivalence classes (a-f) consists of two
small equivalence classes. Consequently, there should be 12
small classes in \cite{GK:JCMCC:2002} for the first parameter set. However, only 9 small classes have been listed in that
paper. Similarly, there are 28 small classes for the second
parameter set, but only 22 are listed in \cite{GK:JCMCC:2002}.

\subsection{ $v=35$ \label{35} }

We have verified the claim made in \cite{GKS:CompStat:2002}
that, for the type (ksss), there are 4 and 2 equivalence classes
for the GS-parameter sets $(35;17,16,16,12;26)$ and
$(35;17,16,14,13;25)$, respectively.

\subsection{ $v=37$ \label{37} }

For (ksss) type, we have verified the claim made in
\cite{GKS:CompStat:2002} that there is only one
equivalence class of solutions for the parameter set
$(37;18,18,16,13;28)$ and only one for
$(37;18,15,15,15;26)$.

For (kkss) type, the family (a) below was constructed in
\cite{Djokovic:OperMatr:2009}
and the families (b-e) in \cite{GK:JCMCC:2002}. We have completed
the classification by constructing the families (f) and (g).

\begin{verbatim}
(kkss) type:

(a) (37;18,18,16,13;28)
[1,2,3,4,6,7,8,11,14,16,17,18,22,24,25,27,28,32],
[1,3,6,8,12,15,19,20,21,23,24,26,27,28,30,32,33,35],
[1,5,8,9,14,16,17,18,19,20,21,23,28,29,32,36],
[0,4,5,6,10,12,13,24,25,27,31,32,33].

(b) (37;18,18,16,13;28)
[1,2,3,4,5,13,18,20,21,22,23,25,26,27,28,29,30,31],
[1,2,4,6,8,9,12,14,15,18,20,21,24,26,27,30,32,34],
[1,5,6,7,8,11,15,16,21,22,26,29,30,31,32,36],
[0,1,4,6,9,13,17,20,24,28,31,33,36].

(c) (37;18,18,16,13;28)
[1,4,5,6,7,8,13,17,19,21,22,23,25,26,27,28,34,35],
[1,3,5,7,8,10,11,14,15,16,19,20,24,25,28,31,33,35],
[2,4,5,6,8,13,16,18,19,21,24,29,31,32,33,35],
[0,7,12,13,14,15,18,19,22,23,24,25,30].

(d) (37;18,18,16,13;28)
[1,4,7,10,11,12,15,17,19,21,23,24,28,29,31,32,34,35],
[1,3,4,7,9,11,12,16,17,19,22,23,24,27,29,31,32,35],
[3,7,8,14,15,16,17,18,19,20,21,22,23,29,30,34],
[0,5,7,9,10,11,18,19,26,27,28,30,32].

(e) (37;18,18,16,13;28)
[1,2,4,5,6,7,11,14,16,17,18,22,24,25,27,28,29,34],
[1,2,3,5,10,13,15,19,20,21,23,25,26,28,29,30,31,33],
[2,3,5,6,8,9,10,14,23,27,28,29,31,32,34,35],
[0,3,4,10,12,16,18,19,21,25,27,33,34].

(f) (37;18,18,16,13;28)
[1,2,3,4,5,6,8,9,10,12,14,17,19,21,22,24,26,30],
[1,2,3,4,6,8,10,11,14,15,18,20,21,24,25,28,30,32],
[3,5,6,11,13,14,17,18,19,20,23,24,26,31,32,34],
[0,3,10,11,12,15,16,21,22,25,26,27,34].

(g) (37;18,18,16,13;28)
[1,2,3,4,6,7,9,10,11,14,16,18,20,22,24,25,29,32],
[1,2,3,4,5,8,9,10,11,13,14,16,18,20,22,25,30,31],
[2,3,4,6,7,15,16,18,19,21,22,30,31,33,34,35],
[0,2,5,8,12,13,18,19,24,25,29,32,35].

\end{verbatim}

\subsection{ $v=39$ \label{39} }

For (ksss) type, we have verified the claim made in
\cite{GKS:CompStat:2002} that there are 3
equivalence classes of solutions for the parameter set
$(39;19,18,17,14;29)$ and two classes for
$(39;19,17,16,15;28)$.

The classification of the GS-difference families for
$v=41,43,45,47,49$ and symmetry types (ksss), (kkss),
(kkks) was not carried out so far. The results of
our computations are given below.

\subsection{ $v=41$ \label{41} }

For (ksss) type and the parameter set $(41;20,20,16,16;31)$
our exhaustive search did not find any solutions.
For (kkss) type, there are 3 equivalence classes with
representatives:

\begin{verbatim}
(kkss) type:

(a) (41;20,20,16,16;31)
[2,4,7,8,10,12,13,14,16,19,20,23,24,26,30,32,35,36,38,40],
[3,4,8,9,10,12,16,19,21,23,24,26,27,28,30,34,35,36,39,40],
[4,5,6,7,9,13,14,20,21,27,28,32,34,35,36,37],
[8,9,11,12,13,15,18,20,21,23,26,28,29,30,32,33].

(b) (41;20,20,16,16;31)
[2,4,6,7,8,10,11,14,15,17,20,22,23,25,28,29,32,36,38,40],
[1,2,4,6,9,10,14,16,17,18,19,20,26,28,29,30,33,34,36,38],
[3,4,5,6,7,10,11,16,25,30,31,34,35,36,37,38],
[2,4,7,8,9,12,14,15,26,27,29,32,33,34,37,39].

(c) (41;20,20,16,16;31)
[2,7,9,11,13,14,16,17,20,22,23,26,29,31,33,35,36,37,38,40],
[1,2,7,9,11,12,15,17,18,19,20,25,27,28,31,33,35,36,37,38],
[5,6,7,9,13,14,17,18,23,24,27,28,32,34,35,36],
[2,3,4,6,7,9,10,18,23,31,32,34,35,37,38,39].

\end{verbatim}

The GS-difference family (b) above is equivalent to the
one found in \cite{GK:JCMCC:2002}. Our exhaustive search
found two more equivalence classes. Note that the
second family listed in \cite{GK:JCMCC:2002} for $n=41$
is not a difference family.

\subsection{ $v=43$ \label{43} }

There are four GS-parameter sets with $v=43$ and
$k_1=21\ge k_2\ge k_3\ge k_4\ge0$. Three of them have
no GS-difference families of type (ksss). Only the
GS-parameter set $(43;21,19,19,16;32)$ has such
families, they form just one equivalence class.

\begin{verbatim}
(ksss) type:

(a) (43;21,19,19,16;32)
[2,3,4,5,6,7,10,12,14,15,16,19,20,21,25,26,30,32,34,35,42],
[0,2,4,5,7,10,11,12,13,17,26,30,31,32,33,36,38,39,41],
[0,3,5,8,9,15,17,18,19,21,22,24,25,26,28,34,35,38,40],
[4,5,6,8,12,13,16,19,24,27,30,31,35,37,38,39].

(kkss) type:

(a) (43;21,21,21,15;35)
[1,3,5,9,11,13,14,16,17,21,23,24,25,28,31,33,35,36,37,39,41],
[4,5,6,9,14,15,16,17,18,19,21,23,30,31,32,33,35,36,40,41,42],
[0,1,3,4,6,7,8,10,14,15,20,23,28,29,33,35,36,37,39,40,42],
[0,1,4,5,6,14,17,20,23,26,29,37,38,39,42].

(b) (43;21,21,18,16;33)
[1,4,5,7,10,12,15,17,19,21,23,25,27,29,30,32,34,35,37,40,41],
[1,2,3,10,13,14,17,19,22,23,25,27,28,31,32,34,35,36,37,38,39],
[2,3,4,7,9,10,12,13,17,26,30,31,33,34,36,39,40,41],
[3,4,8,12,13,14,15,16,27,28,29,30,31,35,39,40].

(c) (43;21,21,18,16;33)
[2,4,6,9,10,15,16,18,21,23,24,26,29,30,31,32,35,36,38,40,42],
[4,7,8,9,11,12,15,17,22,23,24,25,27,29,30,33,37,38,40,41,42],
[6,7,9,13,14,15,17,18,19,24,25,26,28,29,30,34,36,37],
[7,8,9,11,13,16,17,20,23,26,27,30,32,34,35,36].

(d) (43;21,21,18,16;33)
[1,3,8,10,12,15,16,18,19,21,23,26,29,30,32,34,36,37,38,39,41],
[1,2,4,5,6,10,11,12,13,14,16,17,19,20,21,25,28,34,35,36,40],
[2,4,6,7,12,14,15,18,19,24,25,28,29,31,36,37,39,41],
[3,6,7,13,18,19,20,21,22,23,24,25,30,36,37,40].

(kkks) type:

(a) (43;21,21,21,15;35)
[3,5,6,8,10,13,14,15,19,21,23,25,26,27,31,32,34,36,39,41,42],
[3,4,7,11,13,14,17,19,20,22,25,27,28,31,33,34,35,37,38,41,42],
[1,2,3,13,14,16,17,19,20,21,25,28,31,32,33,34,35,36,37,38,39],
[0,6,10,14,15,16,17,19,24,26,27,28,29,33,37].

(b) (43;21,21,21,15;35)
[1,3,4,6,9,10,13,14,19,21,23,25,26,27,28,31,32,35,36,38,41],
[1,3,7,8,10,13,14,15,17,18,19,21,23,27,31,32,34,37,38,39,41],
[2,4,5,6,13,14,15,16,17,20,21,24,25,31,32,33,34,35,36,40,42],
[0,1,2,4,5,10,12,18,25,31,33,38,39,41,42].

(c) (43;21,21,21,15;35)
[1,2,3,5,10,13,16,17,19,21,23,25,28,29,31,32,34,35,36,37,39],
[1,2,4,5,6,10,11,13,14,16,19,21,23,25,26,28,31,34,35,36,40],
[2,4,5,6,7,10,11,12,13,17,21,23,24,25,27,28,29,34,35,40,42],
[0,1,2,3,6,10,16,17,26,27,33,37,40,41,42].

(d) (43;21,21,21,15;35)
[2,4,5,6,8,9,11,12,13,17,19,21,23,25,27,28,29,33,36,40,42],
[2,5,6,7,8,11,13,15,16,17,18,20,21,24,29,31,33,34,39,40,42],
[4,5,6,9,10,11,13,14,16,17,18,19,20,21,28,31,35,36,40,41,42],
[0,3,10,11,13,14,15,20,23,28,29,30,32,33,40].

(e) (43;21,21,21,15;35)
[1,3,6,8,11,12,15,17,21,23,24,25,27,29,30,33,34,36,38,39,41],
[6,8,9,10,11,13,16,17,20,21,24,25,28,29,31,36,38,39,40,41,42],
[2,3,8,9,10,11,12,13,15,16,17,18,19,20,22,29,36,37,38,39,42],
[0,2,4,8,9,14,15,20,23,28,29,34,35,39,41].

\end{verbatim}

\subsection{ $v=45$ \label{45} }

\begin{verbatim}
(ksss) type:

(a) (45;22,22,21,16;36)
[1,2,3,6,7,10,11,13,15,22,24,25,26,27,28,29,31,33,36,37,40,41],
[1,3,4,7,9,10,13,15,20,21,22,23,24,25,30,32,35,36,38,41,42,44],
[0,2,3,4,5,9,12,13,15,17,22,23,28,30,32,33,36,40,41,42,43],
[1,6,14,15,17,18,21,22,23,24,27,28,30,31,39,44].

(b) (45;22,22,21,16;36)
[1,2,4,8,9,10,12,14,16,18,19,21,22,25,28,30,32,34,38,39,40,42],
[1,2,4,5,7,10,14,15,17,21,22,23,24,28,30,31,35,38,40,41,43,44],
[0,1,4,5,6,9,10,11,12,14,22,23,31,33,34,35,36,39,40,41,44],
[2,5,9,10,18,20,21,22,23,24,25,27,35,36,40,43].

(c) (45;22,21,19,17;34)
[1,3,6,7,9,12,16,18,21,22,25,26,28,30,31,32,34,35,37,40,41,43],
[0,1,2,4,8,9,10,11,12,15,19,26,30,33,34,35,36,37,41,43,44],
[0,2,3,5,10,13,14,18,19,20,25,26,27,31,32,35,40,42,43],
[0,3,4,5,7,9,17,21,22,23,24,28,36,38,40,41,42].

(d) (45;22,19,19,18;33)
[1,3,6,8,10,12,18,21,22,25,26,28,29,30,31,32,34,36,38,40,41,43],
[0,1,7,8,11,13,18,19,21,22,23,24,26,27,32,34,37,38,44],
[0,1,2,4,5,12,13,17,19,22,23,26,28,32,33,40,41,43,44],
[3,4,5,6,8,11,12,17,21,24,28,33,34,37,39,40,41,42].

(kkss) type:

(a) (45;22,22,21,16;39)
[1,2,5,7,8,12,13,15,17,19,21,22,25,27,29,31,34,35,36,39,41,42],
[2,3,4,5,6,9,11,12,13,17,18,20,21,22,26,29,30,31,35,37,38,44],
[0,1,2,3,4,6,10,11,13,19,21,24,26,32,34,35,39,41,42,43,44],
[3,4,7,8,17,19,20,22,23,25,26,28,37,38,41,42].

\end{verbatim}

\subsection{ $v=47$ \label{47} }

In this case there are two GS-parameter sets with $k_1=(v-1)/2$,
namely $(47;23,22,22,17;37)$ and $(47;23,21,19,19;35)$.
Our exhaustive search did not find any families of type (ksss).

\subsection{ $v=49$ \label{49} }

In this case there are again two GS-parameter sets with
$k_1=(v-1)/2$, namely $(49;24,24,22,18;39)$ and $(49;24,22,21,19;37)$.
Our exhaustive search did not find any families of type (ksss)
or (kkss).

\end{document}